\documentclass[12pt,reqno]{amsart}

\usepackage{amssymb}
\usepackage{amsmath}
\usepackage{latexsym}
\usepackage{amsthm}
\usepackage{epsfig}
\usepackage{amsfonts}
\usepackage{amssymb}
\usepackage{amsmath}
\usepackage{latexsym}
\usepackage{amsthm}
\usepackage{epsfig}
\usepackage{color}
\usepackage{amsfonts,amssymb,mathrsfs,amscd}
\usepackage{comment}
\usepackage{bm}
\newtheorem{proposition}{Proposition}[section]
\newtheorem{theorem}[proposition]{Theorem}

\theoremstyle{definition}
\newtheorem{definition}[proposition]{Definition}
\newtheorem{remark}[proposition]{Remark}

\numberwithin{equation}{section}
\renewcommand{\div}{\mathop\mathrm{div}}

\begin{document}
\title
[Navier--Stokes--Voigt system in $\mathbb R^4$]
 {Attractor of the  limiting Navier--Stokes--Voigt system in $\mathbb R^4$}

 \author[ A. Ilyin, V. Kalantarov
and  S. Zelik] { Alexei Ilyin${}^1$, Varga Kalantarov${}^{2,3}$ and Sergey
Zelik${}^{1,4,5,6}$}

 % Enter 2010 Mathematics Subject Classification.


\email{ilyin@keldysh.ru}\email{vkalantarov@ku.edu.tr}\email{s.zelik@surrey.ac.uk}
\address{${}^1$ Keldysh Institute of Applied Mathematics, Moscow, Russia}
\address{${}^2$ Ko\c{c} University, Istanbul, Turkey}
\address{${}^3$ Azerbaijan Technical University, Baku, Azerbaijan}
\address{${}^4$ Zhejiang Normal University, Department of Mathematics, Zhejiang, China}
\address{${}^5$ University of Surrey, Department of Mathematics, Guildford, GU2 7XH, United
Kingdom}
\address{${}^6$ HSE University, Nizhny Novgorod, Russia}

\begin{abstract}
The  Navier--Stokes--Voigt system in the whole four-dimensional
space is considered. Although we do not know any physical reasons
to consider this system in space dimension four, the attractors
theory for this case becomes especially simple and elegant and
nothing similar happens when the space dimension is different than
four. These notes are devoted to developing this theory, including
well-posedness, dissipativity, existence of a global attractor and
estimates for its dimension.
\end{abstract}

\maketitle
 \setcounter{equation}{0}

\section{Statement of the problem and main result} \label{sec1}

The last decades have seen considerable interest in various
regularized models in incompressible hydrodynamics.  These models
have much better properties as far as the problem
of wellposedness is concerned. In practice, they can be used
in large scale modeling, where the small spatial scales
are naturally filtered out.

Without claiming to be complete, we point out the works
\cite{ChepTitiVishik,CHOT,CotiGal,FHT, TitiVarga}
and the references therein. Typical
problems from the point of view of attractors  are the
construction of the attractor of the
regularized system,
estimates of its dimension and the study of  the singular limits,
in other words, proving   its convergence
to a weak attractor of the initial three-dimensional system.

One of these models is the Navier--Stokes--Voigt model:
\begin{equation*}\label{DEalpha}
\left\{
  \begin{array}{ll}
    (1-\alpha \Delta)\partial_t u+( u,\nabla u) u-\nu\Delta u+\nabla p=g,\  \  \\
    \operatorname{div}  u=0,\quad u(0)=u_0,\\
%u\vert_{\partial\Omega}=0,
  \end{array}
\right.
\end{equation*}
where $g$ is the given forcing term, and $\nu$ is the kinematic
viscosity coefficient, and system is studied either under Dirichlet
or periodic boundary conditions in dimension $d=2,3$. The parameter
$\alpha$ has the dimension of length squared   and determines the
scale at which high-frequency spatial modes will be filtered out.

The well-posedness of this system goes back to \cite{Oskol}. The
best to date estimates for the dimension of the attractor were
obtained recently in \cite{IZLob}, improving the previously known
results in \cite{TitiVarga} and  \cite{CotiGal} in the 3D case,
while in the 2D case it was shown that the estimates so obtained
converge as $\alpha\to0$ to the well-known estimates for the
classical Navier--Stokes equations both  under  the Dirichlet and
periodic boundary conditions. The key technical tool there were
inequalities for systems of functions with orthonormal derivatives
\cite{LiebJFA} that were further developed in \cite{IKZ22,IZLap70,
IKZMZ22,ZIMS23} and made it possible to obtain optimal bounds for the
attractor dimension for a similar model, namely, the regularized
damped   Euler system both in 2D and 3D, as well as for another
classical example in the theory, namely, the weakly damped
nonlinear hyperbolic system~\cite{IKZIZV}.

Thus, the Voigt regularization of the  Navier--Stokes equations is becoming
more and more popular nowadays and there is a vast number of works
devoted to studying this regularization in the limit $\alpha\to~0$
from different points of view
 including the one, related with the attractors theory.

 Of course, the most interesting from the point of view of
applications are the 3D or 2D cases, but the higher dimensional
cases may be interesting from the mathematical point of view. One
of such results is presented here.

Namely, as we discovered, the Navier--Stokes--Voigt in
$\mathbb R^4$ has ideal properties which allows us to get the main results
of the attractors theory (including the finite-dimensionality and
effective estimates for the dimension of the attractor) in a very
transparent way. Actually, we do not know how to get anything
similar in the case of $\mathbb R^d$, $d\ne4$.

Furthermore, instead of the limit $\alpha\to0$ in this work we are
interested in the opposite case when $\alpha\to\infty$. So after
rescaling the time we obtain the system
\begin{equation}\label{1}
\left\{
  \begin{array}{ll}
    -\Delta \partial_t u+(u,\nabla )u+\nabla  p=
\nu\Delta u+g,\  \  \\
    \operatorname{div}  u=0,\quad u(0)=u_0,\\
  \end{array}
\right.
\end{equation}
where $x\in\Omega=\mathbb R^4$.

The phase space the system is studied in is the homogeneous
Sobolev space $\dot H^1_\sigma$ of divergence free vector functions
with the norm
$$
\|u\|_{\dot H^1_\sigma}=\|\nabla u\|_{L^2}.
$$
The main result of these short notes is the existence of a  a global
attractor $\mathscr A\subset~\dot H^1_\sigma$ for the semigroup associated with equation \eqref{1} and the following upper bound for its fractal dimension. 
\begin{theorem}\label{Th:3}
 Let $g\in\dot H^{-1}_\sigma$.
 Then the fractal dimension of the attractor $\mathscr A$
 associated with problem \eqref{1} is finite
 and possesses the following upper bound:
\begin{equation}\label{10}
\dim_f\mathscr A\le\frac{12L_{0,4}\|g\|^2_{\dot H^{-1}}}{\nu^4}
\le 0.23\cdot\frac{\|g\|^2_{\dot H^{-1}}}{\nu^4} ,
\end{equation}
where $L_{0,4}$ is the constant in the Cwikel--Lieb--Rosenblum
inequality.
\end{theorem}

 \setcounter{equation}{0}
\section{Attractor of the Navier--Stokes--Voigt system } \label{sec2}
Note that the formal multiplication of \eqref{1}  by $u$ followed
by integration by parts gives
\begin{equation}\label{2}
\frac12\frac d{dt}\|\nabla u\|^2_{L^2}+\nu\|\nabla u\|^2_{L^2}=
(g,u)\le\frac1{2\nu}\|g\|^2_{\dot H^{-1}_\sigma}+\frac\nu2\|\nabla u\|^2_{L^2},
\end{equation}
where
$\dot H^{-1}_\sigma$ is the dual
space to $\dot H^1_\sigma$. Crucial for what follows is the
Sobolev  inequality:
\begin{equation}\label{3}
\dot H^1_\sigma\hookrightarrow L^4,\ \ \|u\|_{L^4}\le C_{Sob}\|\nabla u\|_{L^2},
\end{equation}
as well as its highly non-trivial generalization
to  systems of functions with orthonormal gradients in Theorem~\ref{T:Lieb_suborth}.
From the energy identity \eqref{2}, we get the control of the
$\dot H^1_\sigma$-norm of a solution:
\begin{equation}\label{4}
\|\nabla u(t)\|^2_{L^2}\le \|\nabla u_0\|^2_{L^2}e^{-\nu t}+
\frac{1-e^{-\nu t}}{\nu^2}\|g\|^2_{\dot H^{-1}_\sigma}.
\end{equation}
This suggests how to define an energy solution for problem
\eqref{1}. Namely, a function $u$ is a weak energy solution of
problem \eqref{1} if $u\in C([0,T],\dot H^1_\sigma)$ and
\eqref{1} is satisfied in the sense of solenoidal distributions,
namely, for every $\varphi\in C_{0,\sigma}^\infty([0,T]\times\mathbb R^4)$,
\begin{multline}\label{5}
\int_{\mathbb R^4}(\nabla u(t),\partial_t\nabla\varphi(t))\,dt+
\nu\int_{\mathbb R^4}(\nabla u(t),\nabla \varphi(t))\,dt+\\+
\int_{\mathbb R^4}((u(t),\nabla)u(t),\varphi(t))\,dt=
\int_{\mathbb R^4}(g,\varphi(t))\,dt.
\end{multline}
In particular, equation \eqref{1} can be written as an identity in
the space
$C([0,T],\dot H^{-1}_\sigma)$ as follows:
\begin{equation}\label{6}
-\Delta \partial_t u+\nu\Delta u+P(u,\nabla)u=Pg,
\end{equation}
where $P$ is the Leray--Helmholtz projector. Indeed, using the
 formula $(u,\nabla)u=\div(u\otimes u)$ and embedding
\eqref{3}, we see that $u\otimes u\in L^2$ and therefore,
$(u,\nabla)u\in \dot H^{-1}_\sigma$, so all terms in \eqref{6} belong
to $\dot H^{-1}_\sigma$. In turn, this justifies the multiplication
of equation \eqref{6} by $u$, so any weak energy solution of
\eqref{1} satisfies the energy identity \eqref{2}.
\par
We are now ready to discuss the wellposedness of the  problem \eqref{1}.
\begin{theorem}\label{Th:1} Let $g\in\dot H^{-1}_\sigma$. Then, for every
$u_0\in\dot H^1_\sigma$, there exists a unique energy solution
$u(t)$ of the problem \eqref{1}. Moreover, this problem generates a
dissipative semigroup $S(t)$ in the phase space $\dot H^1_\sigma$.
\end{theorem}
\begin{proof} The existence of a solution can be verified in a
standard way using, for instance,  the proper modification of the Galerkin
method or Leray-$\alpha$ approximations (for which the existence is
obvious since they actually are  ODEs in $\dot H^1_\sigma$), so we
leave these standard arguments to the reader. Dissipativity would
follow from estimate \eqref{4} (which guarantees the existence of
the proper absorbing ball in the phase space) if we prove the
uniqueness.
\par
Let $u_1(t)$ and $u_2(t)$ be two energy solutions. Then
the difference $v(t):=u_1(t)-u_2(t)$ solves
\begin{equation*}\label{7}
-\partial_t \Delta v+P((u_1(t),\nabla)v+(v,\nabla)u_2(t))=\nu\Delta v,
\end{equation*}
and the multiplication by $v$ gives
\begin{multline*}
\frac12\frac d{dt}\|\nabla v(t)\|^2_{L^2}+\nu\|\nabla v(t)\|^2_{L^2}\le
 \|\nabla u_2(t)\|_{L^2}\|v\|^2_{L^4}\\
 \le C_{Sob}^2\|\nabla u_2(t)\|_{L^2}\|\nabla v\|^2_{L^2},
\end{multline*}
and the Gronwall inequality together with the control for the
solution $u_2$ provided by estimate \eqref{4} gives the desired
uniqueness as well as  Lipschitz continuity with respect to the
initial data. This finishes the proof of the theorem.
\end{proof}

Note that, arguing analogously, we may prove
that the solution semigroup $S(t)$ is not only Lipschitz continuous,
but also $C^\infty$ with respect to the initial data.
\par
Turning to attractors we recall the classical definition
\cite{BV, Tem,ZUMN}.
\begin{definition} A set
$\mathscr A$ is a global attractor for the semigroup
$S(t)$ acting in a phase space $\dot H^1_\sigma$ if
\par
1) The set $\mathscr A\subset\dot H^1_\sigma$ is compact
 and is strictly invariant:
 $$
 S(t)\mathscr A=\mathscr A,\quad t\ge0;
 $$
\par
2) It attracts the images of all bounded sets
in $\dot H^1_\sigma$ as time goes to infinity,
that is, for every bounded set $B$ and every
neighbourhood $\mathscr O(\mathscr A)$ there exists a
time moment $T=T(\mathscr O,B)$ such that
$$
S(t)B\subset\mathscr O(\mathscr A)
$$
for all $t\ge T$.
\end{definition}
\begin{theorem}\label{Th:2}
Let $g\in\dot H^{-1}_\sigma$. Then the solution semigroup $S(t)$
associated with the problem \eqref{1}
possesses a global attractor $\mathscr A$ in the phase
space $\dot H^1_\sigma$.
\end{theorem}
\begin{proof}
 According to the general theory, to establish the existence of a global
 attractor, we need to check the dissipativity (done in Theorem~\ref{Th:1}),
 the continuity with respect to
 the initial data (also done in Theorem~\ref{Th:1}) and the asymptotic compactness.
 The latter means that for every sequence $u_0^n$ of the initial data
 belonging to the absorbing ball and every sequence $t_n\to\infty$,
 the sequence $S(t_n)u_0^n$ is precompact in $\dot H^1_\sigma$.
\par
Since the energy identity \eqref{2} is verified, the asymptotic
compactness can be established in a standard way by using the so-called
energy method, see, for instance,  \cite{MoiseRosaWang, Rosa, Tem}. Thus, the theorem is proved.
\end{proof}
We are now ready to turn to the main result of this work, namely,
the upper bounds for the dimension of the attractor. We recall that
for each $\varepsilon>0$ any compact set $\mathscr A\subset\dot
H^1_\sigma$ can be covered by a finite number of
$\varepsilon$-balls in $\dot H^1_\sigma$. Let $N(\varepsilon)$ be
the minimal number of such balls. Then the (upper) fractal
dimension of the set $\mathscr A$ is defined via
\begin{equation*}\label{9}
\dim_f\mathscr A:=\limsup_{\varepsilon\to0}
\frac{\log_2 N(\varepsilon)}{\log_2\frac1\varepsilon},
\end{equation*}
see \cite{Tem,Rob} for the details.

We are now in position to prove the main result of this short note,
namely, Theorem~\ref{Th:3}

\begin{proof}[Proof of Theorem~\ref{Th:3}]
 We will use the volume contraction method for
 estimating the dimension of the attractor. This method is
 based on the nontrivial fact that if the  linearization of the  $C^1$-smooth
 semigroup contracts infinitesimal $n$-dimensional volumes
 on the attractor, then the fractal dimension of the attractor
 is less than $n$, see \cite{Ch-I2001,Ch-I,Tem} for the details.
\par
 To estimate the volume contraction factor, a version of
 the Liouville formula is used for the equation of variations
 associated with a trajectory $u(t)$ on the attractor:
 $$
 \partial_t v=(-\Delta)^{-1}P(-(u(t),\nabla )v-(v,\nabla) u(t)+\nu\Delta v)=:\mathcal L_{u(t)}v.
 $$
 Then, by the Liouville theorem, the $n$-dimensional contraction
 factor on the attractor possesses the following estimate:
 $$
 \omega_n(S(t),\mathscr A)\le e^{q(n)},\ \
 q(n):=\sup_{u(0)\in\mathscr A}
 \limsup_{T\to\infty}\frac1T\int_0^T\operatorname{Tr}_n\mathcal L_{u(t)}\,dt,
 $$
where the $n$-dimensional trace is defined as follows:
 $$
 \operatorname{Tr}_n\mathcal L_{u(t)}:=
 \sup_{\{v_i\}_{i=1}^n}\sum_{i=1}^n (\nabla\mathcal L_{u(t)}v_i,\nabla v_i),
 $$
 and the supremum is taken with respect to all
 orthonormal systems in $\dot H^1_\sigma$:
 $(\nabla v_i,\nabla v_j)=\delta_{ij}$,
 see \cite{Tem},
so that
 $$
 q(n^*)<0\Rightarrow\dim_f\mathscr A<n^*.
$$

   Thus, we only need to estimate the above $n$-dimensional
 trace. Indeed, using the orthonormality of the $\nabla v_i$'s, we get
 $$
\aligned
\sum_{i=1}^n (\nabla\mathcal L_{u(t)}v_i,\nabla v_i)=
\sum_{i=1}^n(\mathcal L_{u(t)}v_i,-\Delta v_i)\\=
-\nu n-
\sum_{i=1}^n((v_i,\nabla) u(t),v_i)\le\\\le
-\nu n+c_4\int_{\mathbb R^4}\rho(x)|\nabla u(t,x)|\,dx\\
\le -\nu n+c_4\|\nabla u(t)\|\|\rho\|,
\endaligned
 $$
 where $\rho(x):=\sum\limits_{i=1}^d |v_i(x)|^2$.
 Next, we use \eqref{c4}, \eqref{d3} and \eqref{4} to obtain
$$
q(n)\le n^{1/2}
\left(-\nu n^{1/2}+\frac{2\sqrt{3}L_{0,4}^{1/2}\|g\|_{\dot H^{-1}}}\nu\right),
$$
which proves the  estimate~\eqref{10} and completes  the proof of
the theorem, taking into account  explicit bounds  for the
constants  $L_{0,4}$ and $c_4$ given in the next section.
\end{proof}

 \setcounter{equation}{0}
\section{Inequalities } \label{sec3}
We start with the constant $c_4$ and point out that if $A$ is a
$d\times d$-matrix with zero trace, then~\cite{Lieb, IKZ22}
\begin{equation}\label{c4}
(Av,v)\le c_d\|A\||v|^2,\quad c_d=\sqrt{(d-1)/d}.
\end{equation}

The following result~\cite{LiebJFA} (see also~\cite{IKZIZV}),
namely, the estimate for $\rho$ played
the crucial role  in Theorem~\ref{Th:3}.
\begin{theorem}
\label{T:Lieb_suborth} Let $d\ge3$, $p=d/(d-2)$, and let $\{v_i\}_{i=1}^n\in
\dot H^1_0(\Omega)$, $\Omega\subseteq R^d$ make up a system of vector
functions with orthonormal  gradients in $L_2(\Omega)$:
$(\nabla v_i,\nabla v_j)=\delta_{i\,j}$.
Then  the function
$
\rho(x):=\sum\limits_{j=1}^n|v_j(x)|^2
$
satisfies
\begin{equation}\label{d3}
\|\rho\|_{L_p}\le (dL_{0,d})^{2/d}\frac d{d-2}n^{(d-2)/d},
\end{equation}
where $L_{0,d}$ is the constant in the Cwikel--Lieb--Rozenblum
bound for the number $N(0,-\Delta-V)$ of  negative eigenvalues of
the Schr\"odinger operator $-\Delta-V$, $V(x)\ge0$ in $\mathbb
R^d$, see~\cite{Cwikel,L,R}:
\begin{equation*}\label{CLRbound}
N(0,-\Delta-V)\le L_{0,d}\int_{\mathbb R^d}V(x)^{d/2}dx.
\end{equation*}
\end{theorem}

\begin{remark}\label{R:CLR_constant}
{\rm The constant $L_{0,d}$ is traditionally compared with its
semiclassical lower  bound
$$
L_{0,d}\ge L_{0,d}^{\mathrm{cl}}:=\frac{\omega_d}{(2\pi)^d}\,.
$$
The best to date bound for $L_{0,4}$ is Lieb's bound~\cite{L}
$$L_{0,4}\le 6.034\cdot L_{0,4}^{\mathrm{cl}}=0.0032\dots.
$$
For the recent progress in higher dimensions $d\ge5$ see \cite{Hund}.
}
\end{remark}

\subsection*{Acknowledgement}
This work was supported by Moscow Center of Fundamental and Applied Mathematics,
Agreement with the Ministry
of Science and Higher Education of the Russian Federation, No. 075-15-2025-346.

\end{document}